\documentclass[11pt]{amsart}

\usepackage[article]{robertpreamble}
\usepackage{geometry}
\makeindex 

\begin{document}


\title[On a Problem of M. Kac on Laplace Distributions]{On a Problem of M. Kac on Laplace Distributions}

\author[Robert Koirala]{Robert Koirala}
\address{Department of Mathematics, University of California San Diego}
\email{rkoirala@ucsd.edu}
\date{\today}

\dedicatory{Dedicated to the memory of Dan Stroock with admiration and gratitude.}

\begin{abstract}
    We give counterexamples to a problem of M.~Kac in the \emph{Scottish Book}, which asks whether a certain nonlinear operation on two characteristic functions characterizes Laplace distributions, in analogy with the Cram\'er--L\'evy theorem for Gaussian distributions. We then give an affirmative answer to a refined version of the problem. Finally, we develop a general framework for such characterization problems, construct generalized counterexamples, and pose some open questions.
\end{abstract}

\maketitle 



\section{Introduction}

On September 11, 1938, M.~Kac posed in the \emph{Scottish Book} \cite[Problem 178]{MR3242261}
a characterization problem concerning the class $\{(1+a\xi^2)^{-1}: a>0,\ \xi\in\R\}$ of characteristic functions arising from centered Laplace distributions on $\R$, that is,
probability measures with density $\frac{\beta}{2}e^{-\beta|x|}$, $\beta>0$.

\begin{problem}\label{problem-kac}
    Let
    \begin{align}\label{eq:def-of-phi}
        \Phi(x,y)=\frac{1}{\tfrac{1}{x}+\tfrac{1}{y}-1}.
    \end{align}
    Prove that if
    \begin{align}
        \Phi\left(
            \int_{-\infty}^{+\infty} e^{i\xi t}\,d\sigma_1(t),
            \int_{-\infty}^{+\infty} e^{i\xi t}\,d\sigma_2(t)
        \right)
        =\frac{1}{1+\xi^2},
    \end{align}
    then $\sigma_1$ and $\sigma_2$ are centered Laplace distributions.
\end{problem}

The main purpose of this note is threefold. First, in Section~\ref{sec:counter-examples}, we construct counterexamples to Problem~\ref{problem-kac}. Second, in Section~\ref{sec:affirmative-solution}, we identify additional assumptions on
$\sigma_1$ and $\sigma_2$ under which Problem~\ref{problem-kac} admits an affirmative answer. Third, in Section~\ref{sec:general}, we place Kac's problem into a broader framework of operations $\Phi$ on characteristic functions, discuss several families of examples, and end with some open questions. We also include, in Appendix~\ref{sec:discrete}, a discrete analog of the theory.

Kac's question belongs to the classical theory of characterization and factorization of probability laws via characteristic functions. The guiding example is the Cram\'er--L\'evy theorem,
which asserts that if
\begin{align*}
    \left(\int_{-\infty}^{+\infty} e^{i\xi t}\,d\sigma_1(t)\right)
    \left(\int_{-\infty}^{+\infty} e^{i\xi t}\,d\sigma_2(t)\right)
    = e^{-\xi^2/4},
\end{align*}
then $\sigma_1$ and $\sigma_2$ are Gaussian distributions. See, for instance, \cite{MR1545629} or \cite[Theorem 2.2.1]{MR4573029}. This motivates the following general question.

\begin{question}\label{question:intro}
    If a prescribed operation on two characteristic functions belongs to a family of characteristic functions, must each factor belong to that same family?
\end{question}
From this perspective, Kac's problem asks whether the centered Laplace family admits a factorization property analogous to that of the Gaussian family, but now with multiplication replaced by the nonlinear operation $\Phi$ in \eqref{eq:def-of-phi}. In the setting of Lévy processes, $\Phi$ corresponds to adding two independent processes evaluated at a common independent exponential time. Moreover, if $T\sim \mathrm{Exp}(1)$ is independent of a standard Brownian motion $(B_t)$, then $B_{2aT}$ has characteristic function $(1+a\xi^2)^{-1}$, see Remark~\ref{remark-probabilisitc-interpretation} for details. For background on characteristic functions, decomposition problems, and related characterization questions in the classical setting where the operation is multiplication, see, for example, \cite{MR346969,MR189085,MR428382,MR133854,MR346874,MR810001,MR4573029}.

\subsection*{Acknowledgments}
We are grateful to Bennett Chow for constant support and motivation. We thank Patrick Fitzsimmons for questions that prompted our consideration of Appendix~\ref{sec:discrete}.

\section{Counterexamples}\label{sec:counter-examples}

Let $f_1$ and $f_2$ be characteristic functions on $\R$. Then Problem~\ref{problem-kac} is equivalent to asking whether
\begin{align}\label{eq:defining-equation}
    \frac{1}{f_1(\xi)}+\frac{1}{f_2(\xi)}=2+\xi^2 \qquad (\xi\in\R)
\end{align}
implies that $f_1$ and $f_2$ are characteristic functions of centered Laplace distributions, that is,
\begin{align}\label{eq:laplace}
    f_j(\xi)=\frac{1}{1+a_j\xi^2}
    \qquad\text{for some }a_j>0,\quad a_1+a_2=1,\quad j\in\{1,2\}.
\end{align}

In our first counterexample, the functions $f_j$ are neither real-valued nor even.

\begin{proposition}\label{prop:first-counter-example}
    The functions
    \begin{align}
        f_1(\xi)\coloneqq \left(1-\frac{i\xi}{2}+\frac{\xi^2}{2}\right)^{-1},
        \qquad
        f_2(\xi)\coloneqq \left(1+\frac{i\xi}{2}+\frac{\xi^2}{2}\right)^{-1}
    \end{align}
    are characteristic functions on $\R$ satisfying \eqref{eq:defining-equation}, but neither is of the form \eqref{eq:laplace}. In particular, Problem~\ref{problem-kac} is false as stated.
\end{proposition}

\begin{proof}
    Let $X\sim \mathrm{Exp}(1)$ and $Y\sim \mathrm{Exp}(2)$ be independent exponential random variables. Their characteristic functions are $\varphi_X(\xi)=(1-i\xi)^{-1}$ and $\varphi_Y(\xi)=(1-i\xi/2)^{-1}.$ Since $\varphi_{X-Y}(\xi)=\varphi_X(\xi)\varphi_Y(-\xi),$ we obtain
    \begin{align}\label{eq-x-y}
        \frac{1}{\varphi_{X-Y}(\xi)}
        =\frac{1}{\varphi_X(\xi)\varphi_Y(-\xi)}
        =1-\frac{i\xi}{2}+\frac{\xi^2}{2}
        =\frac{1}{f_1(\xi)}.
    \end{align}
    Similarly, $\varphi_{Y-X}(\xi)=f_2(\xi)$.
    Hence $f_1=\varphi_{X-Y}$ and $f_2=\varphi_{Y-X}$ are characteristic functions that satisfy \eqref{eq:defining-equation}. On the other hand, every function of the form \eqref{eq:laplace} is real-valued and even, whereas neither $f_1$ nor $f_2$ has these properties. Therefore neither $f_1$ nor $f_2$ is of the form \eqref{eq:laplace}.
\end{proof}

The next proposition shows that Problem~\ref{problem-kac} remains false even within the class of real-valued even characteristic functions.

\begin{proposition}\label{prop:first-counter-example-2}
    Fix $0<a<\frac12$. The functions
    \begin{align}\label{eq:counter-example-2}
        f_1(\xi)=\frac{2+a\xi^2}{2(1+a\xi^2)},
        \qquad
        f_2(\xi)=\frac{2+a\xi^2}{2+2\xi^2+a\xi^4}
    \end{align}
    are characteristic functions satisfying \eqref{eq:defining-equation}, but neither is of the form \eqref{eq:laplace}.
\end{proposition}

\begin{proof}
    A direct computation shows that the functions in \eqref{eq:counter-example-2} satisfy \eqref{eq:defining-equation}, and it is clear that neither is of the form \eqref{eq:laplace}. We first observe that $f_1$ is the characteristic function of the probability measure
    \begin{align}
        \frac12\delta_0+\frac{1}{4\sqrt a}e^{-|x|/\sqrt a}\,dx.
    \end{align}
    It remains to prove that $f_2$ is also a characteristic function. A partial fraction decomposition gives
    \begin{align}
        f_2(\xi)=\frac{A_+}{\xi^2+\lambda_-}+\frac{A_-}{\xi^2+\lambda_+},
    \end{align}
    where $A_\pm \coloneqq(2\sqrt{1-2a})^{-1}(\sqrt{1-2a}\pm 1),$ and $\lambda_\pm \coloneqq a^{-1}(1\pm \sqrt{1-2a}).$ Since
    \begin{align}
        \int_{\R} e^{i\xi x}\,\frac{1}{2\sqrt{\lambda}}e^{-\sqrt{\lambda}|x|}\,dx
        =\frac{1}{\xi^2+\lambda},
    \end{align}
    it follows that $f_2$ is the Fourier transform of
    \begin{align}
        p_2(x)
        =\frac{A_+}{2\sqrt{\lambda_-}}e^{-\sqrt{\lambda_-}|x|}
        +\frac{A_-}{2\sqrt{\lambda_+}}e^{-\sqrt{\lambda_+}|x|}.
    \end{align}
    We claim that $p_2\ge 0$. Since $0<a<\frac12$, we have $0<\sqrt{1-2a}<1,$ and therefore $\lambda_+>\lambda_->0,$ $A_+>0,$ and $A_-<0.$ Thus we obtain
    \begin{align}\label{eq:positivity-of-p}
        p_2(x)
        =e^{-\sqrt{\lambda_-}|x|}
        \left(
            \frac{A_+}{2\sqrt{\lambda_-}}
            +\frac{A_-}{2\sqrt{\lambda_+}}
            e^{-(\sqrt{\lambda_+}-\sqrt{\lambda_-})|x|}
        \right)
        \ge
        e^{-\sqrt{\lambda_-}|x|}
        \left(
            \frac{A_+}{2\sqrt{\lambda_-}}
            +\frac{A_-}{2\sqrt{\lambda_+}}
        \right).
    \end{align}
    A straightforward simplification shows that
    \begin{align}
        \frac{A_+}{2\sqrt{\lambda_-}}+\frac{A_-}{2\sqrt{\lambda_+}}
        =
        \frac{(1+\sqrt{1-2a})^{3/2}-(1-\sqrt{1-2a})^{3/2}}{4\sqrt{2}\sqrt{1-2a}}
        \ge 0.
    \end{align}
    Combining this with \eqref{eq:positivity-of-p}, we conclude that $p_2(x)\ge 0$ for all $x\in\R$. Finally, since $f_2(0)=1$, the total mass of $p_2$ is equal to $1$. Therefore $p_2$ is a probability density, and hence $f_2$ is a characteristic function.
\end{proof}

\section{A positive result under geometric infinite divisibility}\label{sec:affirmative-solution}

We now impose additional assumptions on the characteristic functions $f_1$ and $f_2$ which, together with \eqref{eq:defining-equation}, force them to be of the form \eqref{eq:laplace}.

\begin{definition}[{\cite{MR773445}}]
    A characteristic function $f$ is said to be \textit{geometrically infinitely divisible} if for every $p\in(0,1)$ there exists a characteristic function $g_p$ such that
    \begin{align}
        f(\xi)=\frac{p\,g_p(\xi)}{1-(1-p)g_p(\xi)}.
    \end{align}
\end{definition}

\begin{example}
    The function $f(\xi)\coloneqq (1+a\xi^2)^{-1}$ is geometrically infinitely divisible. Indeed, one may take $g_p(\xi)\coloneqq (1+pa\xi^2)^{-1}.$
\end{example}

We shall use the following characterization of geometric infinite divisibility.

\begin{theorem}[{\cite{MR773445}; \cite[Theorem 2.2]{steutel1990set}}]\label{theorem:characterization}
    A characteristic function $f$ is geometrically infinitely divisible if and only if $\exp\left(1-\tfrac{1}{f(\xi)}\right)$ is an infinitely divisible characteristic function.
\end{theorem}

We may now state a positive result for a refined version of Problem~\ref{problem-kac}.

\begin{theorem}\label{theorem:main}
    Let $f_1$ and $f_2$ be geometrically infinitely divisible characteristic functions on $\R$ satisfying \eqref{eq:defining-equation}. Assume that $f_1$ is either real-valued or even. Then $f_1$ and $f_2$ are of the form \eqref{eq:laplace}.
\end{theorem}

\begin{proof}
    For $j\in\{1,2\}$, define
    \begin{align}\label{eq-def-of-psi}
        \psi_j(\xi)\coloneqq \frac{1}{f_j(\xi)}-1.
    \end{align}
    Since each $f_j$ is geometrically infinitely divisible, Theorem~\ref{theorem:characterization} implies that $e^{-\psi_j(\xi)}$ is an infinitely divisible characteristic function. Therefore, for each $j$, the function $-\psi_j$ is a L\'evy--Khintchine exponent, see \cite[Section 3.2]{MR4837608}. Hence there exist $\gamma_j\in\R$, $\sigma_j^2\ge 0$, and a L\'evy measure $\nu_j$ on $\R$ such that
    \begin{align}
        -\psi_j(\xi)
        =i\gamma_j\xi-\frac12\sigma_j^2\xi^2
        +\int_{\R}\Bigl(e^{i\xi x}-1-i\xi x\,\1_{\{|x|\le 1\}}\Bigr)\,\nu_j(dx).
    \end{align}
    Summing these identities and using \eqref{eq:defining-equation}, we obtain
    \begin{align}
        -\xi^2
        =i(\gamma_1+\gamma_2)\xi
        -\frac12(\sigma_1^2+\sigma_2^2)\xi^2
        +\int_{\R}\Bigl(e^{i\xi x}-1-i\xi x\,\1_{\{|x|\le 1\}}\Bigr)\,(\nu_1+\nu_2)(dx).
    \end{align}
    The left-hand side is the L\'evy--Khintchine exponent of the Gaussian characteristic function $e^{-\xi^2}$. Since $\nu_1+\nu_2$ is again a Lévy measure, the right-hand side is again a Lévy–Khintchine representation. By uniqueness of the L\'evy--Khintchine triplet, it follows that
    \begin{align}\label{eq:condition-for-levy-khintchine}
        \gamma_1+\gamma_2=0,
        \qquad
        \nu_1+\nu_2=0,
        \qquad
        \frac12(\sigma_1^2+\sigma_2^2)=1.
    \end{align}
    Since $\nu_1$ and $\nu_2$ are nonnegative measures, the identity $\nu_1+\nu_2=0$ implies that $\nu_1=\nu_2=0.$ Consequently, $\psi_j(\xi)=-i\gamma_j\xi+\frac12\sigma_j^2\xi^2.$ If $f_1$ is real-valued, then so is $\psi_1$, and hence $\gamma_1=0$. If $f_1$ is even, then $\psi_1$ is even, and again $\gamma_1=0$. In either case, \eqref{eq:condition-for-levy-khintchine} yields $\gamma_2=0$. Therefore $\psi_j(\xi)=\frac12\sigma_j^2\xi^2.$ Setting $a_j\coloneqq \frac12\sigma_j^2,$ and using \eqref{eq-def-of-psi}, we conclude that
    \begin{align}
        f_j(\xi)=\frac{1}{1+a_j\xi^2},
        \qquad a_1+a_2=1.
    \end{align}
    Thus $f_1$ and $f_2$ are of the form \eqref{eq:laplace}.
\end{proof}

\begin{remark}
    \begin{enumerate}
        \item A symmetry assumption such as real-valuedness or evenness is essential in Theorem~\ref{theorem:main}. If one assumes only that $f_1$ and $f_2$ are geometrically infinitely divisible and satisfy \eqref{eq:defining-equation}, then the above argument yields only
        \begin{align}
            f_j(\xi)=\frac{1}{1-i\gamma_j\xi+a_j\xi^2},
            \qquad \gamma_1+\gamma_2=0, \qquad a_1+a_2=1.
        \end{align}
        The first counterexample in Proposition~\ref{prop:first-counter-example} is exactly of this form.

        \item If one drops geometric infinite divisibility but retains real-valuedness and evenness, then Proposition~\ref{prop:first-counter-example-2} shows that the conclusion of Theorem~\ref{theorem:main} may still fail.
    \end{enumerate}
\end{remark}

\section{Some generalities}\label{sec:general}

In this section, we present a general framework that extends Theorem~\ref{theorem:main}. Let $\cC$ denote the class of continuous negative definite functions $\psi:\R\to[0,\infty)$ satisfying $\psi(0)=0$. Every $\psi\in\cC$ admits a unique L\'evy--Khintchine representation
\begin{align}\label{eq:symmetric-LK}
    \psi(\xi)
    =
    \frac12 \sigma^2 \xi^2
    +
    \int_{(0,\infty)} \bigl(1-\cos(t\xi)\bigr)\,\nu(dt),
\end{align}
where $\sigma\in\R$ and $\nu$ is a measure on $(0,\infty)$ satisfying $\int_{(0,\infty)} (1\wedge t^2)\,\nu(dt)<\infty.$ See \cite[Theorem 4.15]{MR2978140} or \cite[Section 3.1.3 and Theorem 3.2.7]{MR4837608}.

\begin{definition}\label{definition:indecomposable}
    We say that $\psi\in \cC\setminus\{0\}$ is \textit{indecomposable} if for every decomposition $\psi=\psi_1+\psi_2$ with $\psi_1,\psi_2\in\cC,$ there exists $a\in[0,1]$ such that $\psi_1=a\psi$ and $\psi_2=(1-a)\psi$.
\end{definition}

\begin{proposition}\label{prop:indecomposable-rays}
    The indecomposable elements of $\cC$ are precisely the positive multiples of $\xi^2$ and $1-\cos(t\xi)$ for some $t>0$.
\end{proposition}

\begin{proof}
    Let $\psi\in \cC\setminus\{0\}$ and write
    \begin{align}
        \psi(\xi)=\frac12\sigma^2\xi^2+\int_{(0,\infty)}(1-\cos(t\xi))\,\nu(dt)
    \end{align}
    as in \eqref{eq:symmetric-LK}. Suppose first that $\sigma^2>0$ and $\nu\neq 0$. Then $\psi$ is a nontrivial sum of two nonzero elements of $\cC$, so $\psi$ is not indecomposable. Next, suppose that $\sigma^2=0$ but $\nu$ is not a positive multiple of a Dirac mass. Then there exists a Borel set $A\subset(0,\infty)$ such that both $\nu|_A$ and $\nu|_{A^\complement}$ are nonzero. Hence
    \begin{align}
        \psi(\xi)
        =
        \int_A (1-\cos(t\xi))\,\nu(dt)
        +
        \int_{A^\complement} (1-\cos(t\xi))\,\nu(dt)
    \end{align}
    is a nontrivial decomposition inside $\cC$, so again $\psi$ is not indecomposable. Therefore, if $\psi$ is indecomposable, then either $\nu=0$, in which case $\psi(\xi)=\frac12\sigma^2\xi^2,$ or else $\sigma^2=0$ and $\nu=c\delta_t$ for some $c>0$ and $t>0$, in which case $\psi(\xi)=c(1-\cos(t\xi)).$

    Conversely, we show that these functions are indecomposable. First suppose
    \begin{align}\label{eq:levy-sum}
        c\,\xi^2=\psi_1+\psi_2,
        \qquad
        \psi_1,\psi_2\in\cC,
    \end{align}
    with $c>0$. Writing
    \begin{align}
        \psi_j(\xi)
        =
        \frac12\sigma_j^2\xi^2+\int_{(0,\infty)}(1-\cos(t\xi))\,\nu_j(dt),
        \qquad j=1,2,
    \end{align}
    and using the uniqueness of the L\'evy--Khintchine representation together with \eqref{eq:levy-sum}, we obtain $\nu_1+\nu_2=0$ and $\frac12(\sigma_1^2+\sigma_2^2)=c.$ Since $\nu_1,\nu_2\ge 0$, it follows that $\nu_1=\nu_2=0$. Hence each $\psi_j$ is a nonnegative scalar multiple of $\xi^2$, so $c\,\xi^2$ is indecomposable.

    Now suppose
    \begin{align}
        c(1-\cos(t\xi))=\psi_1+\psi_2,
        \qquad
        \psi_1,\psi_2\in\cC,
    \end{align}
    with $c>0$ and $t>0$. Comparing the L\'evy--Khintchine representations of both sides, we see that $\sigma_1=\sigma_2=0$ and $\nu_1+\nu_2=c\delta_t.$ Since $\nu_1,\nu_2$ are positive measures, each $\nu_j$ must be a nonnegative scalar multiple of $\delta_t$. Hence each $\psi_j$ is a nonnegative scalar multiple of $1-\cos(t\xi)$, so $c(1-\cos(t\xi))$ is indecomposable.
\end{proof}

\begin{lemma}\label{lem:cm-kernel}
    Let $L:[0,\infty)\to(0,1]$ be injective and completely monotone, and assume that $L(0)=1$. Let $\psi\in\cC$ and $a\ge 0$. Then $f_a^\psi(\xi)\coloneqq L(a\psi(\xi))$ is the characteristic function of a probability distribution on $\R$.
\end{lemma}

\begin{proof}
    By the Bernstein--Widder theorem, see \cite[Chapter IV, Theorem 12a]{MR5923}, there exists a probability measure $\rho$ on $[0,\infty)$ such that
    \begin{align}
        L(s)=\int_{[0,\infty)} e^{-st}\,\rho(dt),\qquad 
        f_a^\psi(\xi)=\int_{[0,\infty)} e^{-ta\psi(\xi)}\,\rho(dt).
    \end{align}
    Since $\psi\in\cC$, Schoenberg's theorem implies that $e^{-ta\psi(\xi)}$ is a characteristic function for every $t\ge 0$, see \cite{MR1503439} or \cite[Theorem 7.8]{MR481057}. Therefore $f_a^\psi$ is a convex combination of characteristic functions, and hence is again a characteristic function. Since also $f_a^\psi(0)=L(0)=1$, it is the characteristic function of a probability distribution.
\end{proof}

For such an $L$, set $I_L\coloneqq L([0,\infty))\subset(0,1]$
and define $\Phi_L:I_L\times I_L\to I_L$ by
\begin{align}\label{eq:definition-phi-l}
    \Phi_L(x,y)\coloneqq L\bigl(L^{-1}(x)+L^{-1}(y)\bigr).
\end{align}

\begin{remark}\label{remark-probabilisitc-interpretation}
    Suppose $\eta_1,\eta_2\in \cC$ and $f_j(\xi)=L(\eta_j(\xi)),$ $j=1,2,$ where $L$ is as in Lemma \ref{lem:cm-kernel}. Write
    \begin{align}
        L(s)=\int_{[0,\infty)} e^{-st}\,\rho(dt)
    \end{align}
    for the Bernstein--Widder representation, and let $T$ be a random variable with law $\rho$. If $(Y_j(t))_{t\ge0}$ are independent L\'evy processes with $\bE [e^{i\xi Y_j(t)}]=e^{-t\eta_j(\xi)},$ then
    \begin{align}
        \bE [e^{i\xi Y_j(T)}]=\int_{0}^\infty e^{-t\eta_j(\xi)}\,\rho(dt)=L(\eta_j(\xi))=f_j(\xi).
    \end{align}
    Using the same random time $T$ for both processes, we get
    \begin{align}
        \bE [e^{i\xi (Y_1(T)+Y_2(T))}]=L(\eta_1(\xi)+\eta_2(\xi)) =\Phi_L(f_1(\xi),f_2(\xi)).
    \end{align}
    Thus, on the image of $L$, the operation $\Phi_L$ corresponds to adding independent L\'evy
    processes at a common independent random time. In the special case $L(s)=(1+s)^{-1}$, the law
    of $T$ is $\mathrm{Exp}(1)$, recovering Kac's operation $\Phi$.
\end{remark}

\begin{theorem}\label{thm:binary-characterization}
    Let $L$ be as in Lemma~\ref{lem:cm-kernel}, and let $\psi\in\cC$ be indecomposable in the sense of Definition~\ref{definition:indecomposable}. Suppose that $\eta_1,\eta_2\in\cC$, and set $f_j(\xi)\coloneqq L(\eta_j(\xi)),$ $j=1,2.$
    If
    \begin{align}\label{eq:general-phi-L}
        \Phi_L\bigl(f_1(\xi),f_2(\xi)\bigr)=L(\psi(\xi))
        \qquad\text{for all }\xi\in\R,
    \end{align}
    then there exist $a_1,a_2\ge 0$ with $a_1+a_2=1$ such that
    \begin{align}\label{eq:general-conclusion}
        f_j(\xi)=L(a_j\psi(\xi)),
        \qquad j=1,2.
    \end{align}
\end{theorem}

\begin{proof}
    By the definition of $\Phi_L$, equation \eqref{eq:general-phi-L} becomes
    \begin{align}
        L\bigl(L^{-1}(f_1(\xi))+L^{-1}(f_2(\xi))\bigr)=L(\psi(\xi)).
    \end{align}
    Since $L$ is injective on $[0,\infty)$, it follows that
    \begin{align}
        L^{-1}(f_1(\xi))+L^{-1}(f_2(\xi))=\psi(\xi).
    \end{align}
    Because $f_j(\xi)=L(\eta_j(\xi))$, we obtain
    \begin{align}
        \eta_1(\xi)+\eta_2(\xi)=\psi(\xi)
        \qquad\text{for all }\xi\in\R.
    \end{align}
    Since $\eta_1,\eta_2\in\cC$ and $\psi$ is indecomposable, there exists $a\in[0,1]$ such that
    \begin{align}
        \eta_1=a\psi,
        \qquad
        \eta_2=(1-a)\psi.
    \end{align}
    Setting $a_1=a$ and $a_2=1-a$ proves \eqref{eq:general-conclusion}.
\end{proof}

We now record several special cases of Theorem~\ref{thm:binary-characterization}.

\begin{example}
    Take $L(s)=e^{-s},$ so that $\Phi_L(x,y)=xy.$
    \begin{enumerate}
        \item If $\psi(\xi)=\xi^2$, then $L(\psi(\xi))=e^{-\xi^2},$ which is the characteristic function of a centered Gaussian law. Thus multiplication characterizes the Gaussian family inside $\{e^{-\eta(\xi)}:\eta\in\cC\}.$ This is a special case of Cramér--Lévy theorem, see \cite[Theorem 2.2.1]{MR4573029}.

        \item If $\psi_t(\xi)=a(1-\cos(t\xi))$ with $a>0$ and $t>0$, then $L(\psi_t(\xi))=\exp\bigl(a(\cos(t\xi)-1)\bigr),$ which is the characteristic function of the symmetric compound Poisson law with intensity $a$ and jumps $\pm t$ with probabilities $1/2$. Thus multiplication characterizes this symmetric lattice family inside $\{e^{-\eta(\xi)}:\eta\in\cC\}.$ 
    \end{enumerate}
\end{example}

\begin{example}
    Fix $\alpha>0$ and $0<\beta\le 1$, and define
    \begin{align}
        L_{\alpha,\beta}(s)=(1+s^\beta)^{-\alpha},
        \qquad
        \Phi_{L_{\alpha,\beta}}(x,y)
        =
        \left(
            1+
            \left(
                (x^{-1/\alpha}-1)^{1/\beta}
                +
                (y^{-1/\alpha}-1)^{1/\beta}
            \right)^\beta
        \right)^{-\alpha}.
    \end{align}
    \begin{enumerate}
        \item If $\psi(\xi)=\xi^2$, then $L_{\alpha,\beta}(\psi(\xi))=(1+|\xi|^{2\beta})^{-\alpha},$ which is the characteristic function of the generalized Linnik law. Hence Theorem~\ref{thm:binary-characterization} shows that $\Phi_{L_{\alpha,\beta}}$ characterizes the generalized Linnik family inside $\{(1+\eta(\xi)^\beta)^{-\alpha}:\eta\in\cC\}.$ The case $(\alpha,\beta)=(1,1)$, combined with Theorem~\ref{theorem:characterization}, recovers Theorem~\ref{theorem:main}. Indeed, by Theorem~\ref{theorem:characterization}, a characteristic function $f$ is geometrically infinitely divisible if and only if there exists a function $\psi$ such that $-\psi$ is a L\'evy--Khintchine exponent and $f(\xi)=L_{1,1}(\psi(\xi))=(1+\psi(\xi))^{-1}.$ If, in addition, $f$ is real-valued or even, then $\psi$ is real-valued and continuous negative definite, hence $\psi\in\cC$.

        \item If $\psi_t(\xi)=2r(1-r)^{-2}(1-\cos(t\xi)),$ where $ t>0,$ $r\in[0,1),$
        then $L_{1,1}(\psi_t(\xi)) = (1-r)^2(1-2r\cos(t\xi)+r^2)^{-1},$
        which is the characteristic function of the two-sided geometric law on $t\Z$,
        \begin{align}\label{eq:two-sided}
            \bP(X=kt)=\frac{1-r}{1+r}r^{|k|},
            \qquad k\in\Z.
        \end{align}
        Thus $\Phi_{L_{1,1}}$ characterizes the law \eqref{eq:two-sided} inside $\{(1+\eta(\xi))^{-1}:\eta\in\cC\}.$
    \end{enumerate}
\end{example}

\begin{example}
    Fix $0<\beta\le 1$, and set
    \begin{align}
        L_\beta(s)=e^{-s^\beta},
        \qquad
        \Phi_{L_\beta}(x,y)
        =
        \exp\left(
            -\Bigl((-\log x)^{1/\beta}+(-\log y)^{1/\beta}\Bigr)^\beta
        \right).
    \end{align}
    If $\psi(\xi)=\xi^2$, then $L_\beta(\psi(\xi))=e^{-|\xi|^{2\beta}},$ which is the characteristic function of the centered symmetric \(2\beta\)-stable law. Therefore Theorem~\ref{thm:binary-characterization} shows that $\Phi_{L_\beta}$ characterizes the symmetric \(2\beta\)-stable family inside $\{e^{-\eta(\xi)^\beta}:\eta\in\cC\}.$ In particular, when $\beta=\frac12$, one recovers the Cauchy law.
\end{example}

\begin{remark}
    For additional examples of functions $L$ satisfying the hypotheses of Lemma~\ref{lem:cm-kernel}, we refer the reader to \cite{MR2978140}.
\end{remark}

\subsection{Generalized counterexamples}

We now generalize the counterexamples from Section~\ref{sec:counter-examples}. Fix $\beta>0$, and define
\begin{align}\label{eq:almost-gaussian}
    L_\beta(s)\coloneqq \left(1+\frac{s}{\beta}\right)^{-\beta},
    \qquad
    \Phi_\beta(x,y)\coloneqq \bigl(x^{-1/\beta}+y^{-1/\beta}-1\bigr)^{-\beta}.
\end{align}
In the complex-valued examples below, when $\beta$ is not an integer we interpret
$z\mapsto z^{-1/\beta}$ using the principal branch. This is legitimate because the relevant values remain in the slit plane $\C\setminus(-\infty,0]$. For $a\ge 0$, set $\psi_a(\xi)\coloneqq a\xi^2$ and define
\begin{align}\label{eq:general-counter-example}
    \cG_\beta
    \coloneqq
    \left\{
        L_\beta(\psi_a(\xi))
        =
        \left(1+\frac{a\xi^2}{\beta}\right)^{-\beta}
        \colon a>0
    \right\}.
\end{align}

The first family generalizes Proposition~\ref{prop:first-counter-example}.

\begin{proposition}\label{prop:first-counter-example-general}
    Fix $\beta>0$ and $a_1,a_2>0$. Assume that either $0<\beta\le 2$ and $b\neq 0$ is arbitrary, or that $\beta>2$ and
    \begin{align}\label{eq:branch-condition}
        0<|b|<2\sqrt{\beta}\tan\!\left(\frac{\pi}{\beta}\right)\min\{\sqrt{a_1},\sqrt{a_2}\}.
    \end{align}
    Then the functions
    \begin{align}
        f_1(\xi)\coloneqq \left(1+\frac{a_1\xi^2-ib\xi}{\beta}\right)^{-\beta},
        \qquad
        f_2(\xi)\coloneqq \left(1+\frac{a_2\xi^2+ib\xi}{\beta}\right)^{-\beta}
    \end{align}
    are characteristic functions satisfying
    \begin{align}\label{eq:phi-beta}
        \Phi_\beta\bigl(f_1(\xi),f_2(\xi)\bigr)
        =
        \left(1+\frac{(a_1+a_2)\xi^2}{\beta}\right)^{-\beta}
        \in \cG_\beta,
    \end{align}
    but neither $f_1$ nor $f_2$ belongs to $\cG_\beta$.
\end{proposition}

\begin{proof}
    Let $G$ be a Gamma random variable with shape parameter $\beta$ and scale parameter $1/\beta$. Then
    \begin{align}
        \bE(e^{-zG})
        =
        \left(1+\frac{z}{\beta}\right)^{-\beta},
        \qquad \Re z\ge 0,
    \end{align}
    where the principal branch is understood. Let $Z_1\sim N(0,1)$ be independent of $G$, and define $ X\coloneqq bG+\sqrt{2a_1G}\,Z_1.$ Then
    \begin{align}
        \bE(e^{i\xi X}\mid G) = e^{ibG\xi}\,\bE\bigl(e^{i\xi\sqrt{2a_1G}Z_1}\mid G\bigr) =
        e^{ibG\xi-a_1G\xi^2} = e^{-G(a_1\xi^2-ib\xi)}.
    \end{align}
    Taking expectations gives
    \begin{align}
        \bE(e^{i\xi X})
        =
        \bE\left[e^{-G(a_1\xi^2-ib\xi)}\right]
        =
        f_1(\xi),
    \end{align}
    so $f_1$ is a characteristic function. Similarly, if $Z_2\sim N(0,1)$ is independent of $G$ and $Y\coloneqq -bG+\sqrt{2a_2G}\,Z_2,$ then $\bE(e^{i\xi Y})=f_2(\xi),$ and hence $f_2$ is also a characteristic function. 
    
    Set
    \begin{align}
        w_1(\xi)\coloneqq 1+\frac{a_1\xi^2-ib\xi}{\beta},
        \qquad
        w_2(\xi)\coloneqq 1+\frac{a_2\xi^2+ib\xi}{\beta}.
    \end{align}
    Then $\Re w_j(\xi)=1+\frac{a_j\xi^2}{\beta}>0$, so $w_j(\xi)$ lies in the open right half-plane. If $0<\beta\le 2$, then $|\arg w_j(\xi)|<\pi/2\le \pi/\beta$, and hence the principal branch gives $(w_j(\xi)^{-\beta})^{-1/\beta}=w_j(\xi).$ Assume now that $\beta>2$. Since
    \begin{align}
        |\arg w_j(\xi)|
        =
        \arctan\!\left(\frac{|b||\xi|}{\beta+a_j\xi^2}\right),
    \end{align}
    and the quantity $\frac{|b||\xi|}{\beta+a_j\xi^2}$ is maximized at $|\xi|=\sqrt{\beta/a_j}$ with maximum value $\frac{|b|}{2\sqrt{a_j\beta}}$, condition \eqref{eq:branch-condition} implies
    \begin{align}
        |\arg w_j(\xi)|
        \le
        \arctan\!\left(\frac{|b|}{2\sqrt{a_j\beta}}\right)
        <
        \frac{\pi}{\beta}.
    \end{align}
    Therefore again $(w_j(\xi)^{-\beta})^{-1/\beta}=w_j(\xi).$ Hence $f_1(\xi)^{-1/\beta}=w_1(\xi),$ $f_2(\xi)^{-1/\beta}=w_2(\xi),$ and so \eqref{eq:phi-beta} follows.
    
    Finally, every element of $\cG_\beta$ is real-valued and even, whereas $f_1$ and $f_2$ are not real-valued when $b\neq 0$. Therefore neither $f_1$ nor $f_2$ belongs to $\cG_\beta$.
\end{proof}

The next family generalizes Proposition~\ref{prop:first-counter-example-2}.

\begin{proposition}\label{prop:first-counter-example-2-general}
    For every choice of $a>0$, $n\in\N$, and $\theta\in(0,1/2)$, the functions
    \begin{align}
        f_1(\xi)\coloneqq \left(\frac{1}{2}+\frac{1}{2\left(1+\frac{\theta a}{n}\xi^2\right)}\right)^n,
        \qquad
        f_2(\xi)\coloneqq \left(\frac{2+\frac{\theta a}{n}\xi^2}{2+2\frac{a}{n}\xi^2+\frac{\theta a^2}{n^2}\xi^4}\right)^n
    \end{align}
    are characteristic functions satisfying
    \begin{align}
        \Phi_n\bigl(f_1(\xi),f_2(\xi)\bigr)
        =
        \left(1+\frac{a\xi^2}{n}\right)^{-n}\in \cG_n,
    \end{align}
    but neither $f_1$ nor $f_2$ belongs to $\cG_n$.
\end{proposition}

\begin{proof}
    By Proposition~\ref{prop:first-counter-example-2}, the functions
    \begin{align}
        \widetilde{u}_1(\xi)\coloneqq \frac{1}{2}+\frac{1}{2(1+\theta\xi^2)},
        \qquad
        \widetilde{u}_2(\xi)\coloneqq \frac{2+\theta\xi^2}{2+2\xi^2+\theta\xi^4}
    \end{align}
    are characteristic functions satisfying
    \begin{align}
        \Phi_1\bigl(\widetilde{u}_1(\xi),\widetilde{u}_2(\xi)\bigr)=(1+\xi^2)^{-1}.
    \end{align}
    Rescaling preserves the class of characteristic functions, so $ u_j(\xi)\coloneqq \widetilde{u}_j\!\left(\sqrt{\frac{a}{n}}\,\xi\right),$ $j=1,2,$ are characteristic functions, and
    \begin{align}
        \Phi_1\bigl(u_1(\xi),u_2(\xi)\bigr)
        =
        \left(1+\frac{a\xi^2}{n}\right)^{-1}.
    \end{align}
    Therefore $f_1(\xi)= u_1(\xi)^n$ and $f_2(\xi)= u_2(\xi)^n$ are also characteristic functions. Since $u_1,u_2>0$, we have $f_1(\xi)^{-1/n}=u_1(\xi)^{-1}$ and $f_2(\xi)^{-1/n}=u_2(\xi)^{-1}.$ Hence
    \begin{align*}
        \Phi_n\bigl(f_1(\xi),f_2(\xi)\bigr)
        &=
        \bigl(f_1(\xi)^{-1/n}+f_2(\xi)^{-1/n}-1\bigr)^{-n} =
        \bigl(u_1(\xi)^{-1}+u_2(\xi)^{-1}-1\bigr)^{-n} \\
        &=
        \left(1+\frac{a\xi^2}{n}\right)^{-n}
        \in \cG_n.
    \end{align*}

    It is standard to check that neither $f_1$ nor $f_2$ lies in $\cG_n$. 
\end{proof}

\subsection{The Gaussian limit of counterexamples and some open questions}

For $\beta>0$, define $L_\beta$ and $\Phi_\beta$ as in \eqref{eq:almost-gaussian}. Then $L_\beta(s)\to e^{-s}$ as $\beta\to\infty$ locally uniformly on $[0,\infty)$, and correspondingly $\Phi_\beta$ converges formally to multiplication $\Phi_\infty(x,y)=xy.$ Moreover, every element of $\cG_\beta$ in \eqref{eq:general-counter-example} converges locally uniformly on $\R$ to a Gaussian characteristic function of the form $e^{-a\xi^2}$. On the other hand, Cramér--Lévy theorem implies that if
\begin{align}
    \Phi_\infty(f_1(\xi),f_2(\xi))
    =
    f_1(\xi)f_2(\xi)
    =
    e^{-a\xi^2},
\end{align}
then both $f_1$ and $f_2$ are Gaussian characteristic functions. Thus both families of counterexamples must disappear in the Gaussian limit.

Consider an admissible family $b_\beta$ satisfying \eqref{eq:branch-condition}. Then necessarily $b_\beta\to 0$ as $\beta\to\infty$. Hence, for fixed $a_1,a_2>0$,
\begin{align}
    \left(1+\frac{a_1\xi^2-ib_\beta\xi}{\beta}\right)^{-\beta}
    \to e^{-a_1\xi^2},
    \qquad
    \left(1+\frac{a_2\xi^2+ib_\beta\xi}{\beta}\right)^{-\beta}
    \to e^{-a_2\xi^2}
\end{align}
locally uniformly on $\R$, and these limits are Gaussian characteristic functions.

The second family of counterexamples in Proposition~\ref{prop:first-counter-example-2-general} disappears as $n\to\infty$ as well. Each factor is of the form $f_{j,n}(\xi)=\widetilde u_j\!\left(\sqrt{\frac an}\,\xi\right)^n,$ and hence is the characteristic function of the $\sqrt{a/n}$-rescaled $n$-fold convolution power of the law corresponding to $\widetilde u_j$. Consequently, by a central limit type argument or by a logarithmic expansion,
\begin{align}
    f_{1,n}(\xi)\to e^{-\frac{\theta a}{2}\xi^2},
    \qquad
    f_{2,n}(\xi)\to e^{-\left(1-\frac{\theta}{2}\right)a\xi^2}
\end{align}
locally uniformly on $\R$.

The observation leads naturally to the following questions.

\begin{question}\label{question-1}
    For which injective completely monotone functions $L:[0,\infty)\to(0,1]$ does the following hold? Whenever characteristic functions $f_1,f_2$ with values in $I_L\coloneqq L([0,\infty))$ satisfy
    \begin{align}
        \Phi_L(f_1(\xi),f_2(\xi))=L(\psi(\xi))
    \end{align}
    for some indecomposable $\psi\in \cC$, must there exist $\eta_1,\eta_2\in \cC$ such that $f_j=L(\eta_j),$ $j=1,2?$
\end{question}

\begin{question}
    Is $L(s)=e^{-s}$ characterized among completely monotone functions by the following rigidity property: if
    \begin{align}
        \Phi_L(f_1(\xi),f_2(\xi))=L(a\xi^2)
    \end{align}
    for characteristic functions $f_1$ and $f_2$, then the factors must themselves lie in the image of $L$?
\end{question}

A positive answer to Question~\ref{question-1}, together with Theorem~\ref{thm:binary-characterization}, would force $\eta_1$ and $\eta_2$ to be positive multiples of $\psi$. It would be interesting to understand whether complex-analytic properties of extensions of $L$, or structural properties of the operation $\Phi_L$, play a role here.

More generally, we ask the following.

\begin{question}\label{question-main-3}
    Let $D\subset\C$ and let $\Phi:D\times D\to\C$ be a map such that for every pair of characteristic functions $f_1,f_2$ on $\R$, the quantity $\Phi(f_1(\xi),f_2(\xi))$ is well defined for every $\xi\in\R$. Let $\cM$ be a subfamily of the set of all characteristic functions on $\R$. Suppose that
    \begin{align}
        \Phi(f_1(\xi),f_2(\xi))\in \cM
        \qquad\text{for all }\xi\in\R.
    \end{align}
    Under what conditions can one conclude that $f_1,f_2\in \cM$?
\end{question}

One may also ask for stability versions of the above questions.

\appendix

\section{A discrete analog}\label{sec:discrete}

In this appendix we record a discrete counterpart of the framework in Sections~\ref{sec:counter-examples} and \ref{sec:general}, obtained by replacing characteristic functions by probability generating functions. Recall that the probability generating function (pgf) of an $\mathbb N_0$-valued random variable $X$ is
\begin{align}
    G_X(z)\coloneqq \bE[z^X], \qquad |z|\le 1.
\end{align}
The discrete analog of $\cC$ is
\begin{align}
    \cB\coloneqq \{\eta:[0,\infty)\to[0,\infty): \eta \text{ is Bernstein and } \eta(0)=0\},
\end{align}
see \cite[Definition 3.1]{MR2978140}. Every $\eta\in\cB$ admits a unique representation
\begin{equation}\label{eq:Bernstein-rep}
    \eta(u)=bu+\int_{(0,\infty)}\bigl(1-e^{-us}\bigr)\,\nu(ds),
\end{equation}
where $b\ge 0$ and $\nu$ is a measure on $(0,\infty)$ satisfying $\int_{(0,\infty)} (1\wedge s)\,\nu(ds)<\infty,$ see \cite[Theorem 3.2]{MR2978140}.

\begin{definition}
    We say that $\eta\in \cB\setminus\{0\}$ is \emph{indecomposable} if for every decomposition
    \begin{align}
        \eta=\eta_1+\eta_2,\qquad \eta_1,\eta_2\in\cB,
    \end{align}
    there exists $a\in[0,1]$ such that $\eta_1=a\eta$ and $\eta_2=(1-a)\eta$.
\end{definition}

Using a method similar to that in the proof of Proposition~\ref{prop:indecomposable-rays}, one obtains the following.

\begin{proposition}
    The indecomposable elements of $\cB$ are precisely the positive multiples of $u$ and
    $1-e^{-tu}$ for some $t>0$.
\end{proposition}

\begin{lemma}\label{lem:discrete-L}
    Let $L:[0,\infty)\to(0,1]$ be injective and completely monotone with $L(0)=1$. If $\eta\in\cB$ and $a\ge 0$, then
    \begin{align}
        G_a^\eta(z)\coloneqq L(a\eta(1-z)), \qquad |z|\le 1,
    \end{align}
    is a probability generating function.
\end{lemma}

\begin{proof}
    By the Bernstein--Widder theorem \cite[Chapter IV, Theorem 12a]{MR5923}, there exists a probability measure $\rho$ on $[0,\infty)$ such that
    \begin{align}
        L(s)=\int_{[0,\infty)} e^{-st}\,\rho(dt), \qquad 
        G_a^\eta(z)=\int_{[0,\infty)} e^{-ta\eta(1-z)}\,\rho(dt).
    \end{align}
    Since $\eta$ is Bernstein, the function $u\mapsto e^{-ta\eta(u)}$ is completely monotone for every $t\ge 0$ \cite[Theorem 3.7]{MR2978140}. Therefore, again by Bernstein--Widder, there exists a probability measure $\mu_t$ on $[0,\infty)$ such that
    \begin{align}
        e^{-ta\eta(1-z)}=\int_{[0,\infty)} e^{-s(1-z)}\,\mu_t(ds).
    \end{align}
    For each $s\ge 0$, the function $z\mapsto e^{-s(1-z)}$ is the pgf of a Poisson random variable with mean $s$. Thus $e^{-ta\eta(1-z)}$ is a mixture of pgfs, hence itself a pgf. Averaging once more against $\rho$ shows that $G_a^\eta$ is a pgf.
\end{proof}

For such an $L$, define $\Phi_L$ as in \eqref{eq:definition-phi-l}. The same proof as in Theorem~\ref{thm:binary-characterization} yields the following.

\begin{theorem}\label{thm:discrete-binary}
    Let $L$ be as in Lemma~\ref{lem:discrete-L}, and let $\eta\in\cB$ be indecomposable. Suppose
    $\eta_1,\eta_2\in\cB$, and set $G_j(z)\coloneqq L(\eta_j(1-z)),$ $j=1,2.$ If
    \begin{align}
        \Phi_L(G_1(z),G_2(z))=L(\eta(1-z))
    \end{align}
    for all $z\in[0,1]$, then there exist $a_1,a_2\ge 0$ with $a_1+a_2=1$ such that
    \begin{align}
        G_j(z)=L(a_j\eta(1-z)), \qquad j=1,2.
    \end{align}
\end{theorem}

\begin{example}
    \begin{enumerate}
        \item Take $L(s)=e^{-s},$ so that $\Phi_L(x,y)=xy.$ If $\eta(u)=u,$ then $L(a\eta(1-z))=e^{-a(1-z)},$ $a>0,$ is the pgf of a Poisson law with parameter $a$. Thus multiplication characterizes the Poisson family inside $\{e^{-\eta(1-z)}:\eta\in\cB\}.$ This is a special case of Raikov's theorem. See \cite{raikov1937decomposition,MR428382}.

        \item Take $L(s)=(1+s)^{-1},$ so that $\Phi_L(x,y)=(x^{-1}+y^{-1}-1)^{-1}.$ If $\eta(u)=u,$ then $L(a\eta(1-z))=(1+a(1-z))^{-1},$ $a>0,$ which is the pgf of a geometric law on $\mathbb N_0$. Thus $\Phi_L$ characterizes the geometric family inside $\left\{(1+\eta(1-z))^{-1}:\eta\in\cB\right\}.$
    \end{enumerate}
\end{example}

We next record a discrete analog of Proposition~\ref{prop:first-counter-example-2}.

\begin{proposition}\label{prop:discrete-kac-counterexample}
    Fix $\lambda>0$ and $\theta\in(0,1/2)$. Then the functions
    \begin{align}
        G_1(z)\coloneqq \frac12+\frac{1}{2\bigl(1+\theta\lambda(1-z)\bigr)},
        \qquad
        G_2(z)\coloneqq \frac{2+\theta\lambda(1-z)}
        {2+2\lambda(1-z)+\theta\lambda^2(1-z)^2}.
    \end{align}
    are probability generating functions satisfying
    \begin{align}\label{eq:discrete-phi}
        \Phi(G_1(z),G_2(z))=\frac{1}{1+\lambda(1-z)},
        \qquad
        \Phi(x,y)\coloneqq \frac{1}{x^{-1}+y^{-1}-1},
    \end{align}
    but neither $G_1$ nor $G_2$ is of the form $(1+\mu(1-z))^{-1}$.
\end{proposition}

\begin{proof}
    A direct computation shows that \eqref{eq:discrete-phi} holds. The function $G_1$ is a convex combination of the pgf $1$ of $\delta_0$ and the geometric pgf $(1+\theta\lambda(1-z))^{-1}$, hence it is a pgf.
    
    It remains to prove that $G_2$ is a pgf. Set
    \begin{align}
        r\coloneqq \sqrt{1-2\theta}\in (0,1), \qquad
        p\coloneqq \frac{\lambda(1+r)}{2+\lambda(1+r)}, \qquad
        q\coloneqq \frac{\lambda(1-r)}{2+\lambda(1-r)}.
    \end{align}
    A partial fraction computation gives
    \begin{align}
        G_2(z)=\frac{A}{1-pz}+\frac{B}{1-qz},\qquad 
        A=\frac{1+r-\theta}{r\bigl(2+\lambda(1+r)\bigr)},
        \qquad
        B=\frac{\theta+r-1}{r\bigl(2+\lambda(1-r)\bigr)}.
    \end{align}
    Since $0<q<p<1$, we may expand
    \begin{align}
        G_2(z)=\sum_{n=0}^\infty (Ap^n+Bq^n)z^n.
    \end{align}
    If $B\ge 0$, then all coefficients are nonnegative. If $B<0$, then
    \begin{align*}
        Ap^n+Bq^n
        =p^n\left(A+B\left(\frac{q}{p}\right)^n\right)
        \ge p^n(A+B)=p^nG_2(0)=p^n\frac{2+\theta\lambda}{2+2\lambda+\theta\lambda^2}>0,
    \end{align*}
    so again every coefficient is positive. Since also $G_2(1)=1$, it follows that $G_2$ is a pgf.
    
    Finally, neither $G_1$ nor $G_2$ belongs to the geometric family $\left\{(1+\mu(1-z))^{-1}:\mu>0\right\}.$ Indeed, for such a family one has $G^{-1}-1$ affine in $(1-z)$. But
    \begin{align}
        G_1(z)^{-1}-1=\frac{\theta\lambda(1-z)}{2+\theta\lambda(1-z)}
    \end{align}
    is not affine, and $G_2$ has denominator quadratic in $(1-z)$.
\end{proof}

We also have an analog of Proposition \ref{prop:first-counter-example-2-general}.
\begin{corollary}
    For every $n\in\mathbb N$, $\lambda>0$, and $\theta\in(0,1/2)$, the functions
    \begin{align}
        F_1(z)\coloneqq \left(\frac12+\frac{1}{2\bigl(1+\frac{\theta\lambda}{n}(1-z)\bigr)}\right)^n,
        \qquad
        F_2(z)\coloneqq \left(
        \frac{2+\frac{\theta\lambda}{n}(1-z)}
        {2+\frac{2\lambda}{n}(1-z)+\frac{\theta\lambda^2}{n^2}(1-z)^2}
        \right)^n
    \end{align}
    are pgfs satisfying
    \begin{align*}
        \Phi_n(F_1(z),F_2(z))
        =
        \left(1+\frac{\lambda}{n}(1-z)\right)^{-n},
        \qquad
        \Phi_n(x,y)\coloneqq \bigl(x^{-1/n}+y^{-1/n}-1\bigr)^{-n},
    \end{align*}
    but neither factor belongs to the negative binomial family $\left\{\left(1+\frac{\mu}{n}(1-z)\right)^{-n}:\mu>0\right\}.$
\end{corollary}


\bibliographystyle{amsalpha}
\bibliography{references}

@book {MR481057,
    AUTHOR = {Berg, Christian and Forst, Gunnar},
     TITLE = {Potential theory on locally compact abelian groups},
    SERIES = {Ergebnisse der Mathematik und ihrer Grenzgebiete [Results in
              Mathematics and Related Areas]},
    VOLUME = {Band 87},
 PUBLISHER = {Springer-Verlag, New York-Heidelberg},
      YEAR = {1975},
     PAGES = {vii+197},
   MRCLASS = {31C05 (43A35 60B15)},
  MRNUMBER = {481057},
MRREVIEWER = {Francis\ Hirsch},
}

@article {MR1545629,
    AUTHOR = {Cram\'er, Harald},
     TITLE = {\"Uber eine {E}igenschaft der normalen {V}erteilungsfunktion},
   JOURNAL = {Math. Z.},
  FJOURNAL = {Mathematische Zeitschrift},
    VOLUME = {41},
      YEAR = {1936},
    NUMBER = {1},
     PAGES = {405--414},
      ISSN = {0025-5874,1432-1823},
   MRCLASS = {99-04},
  MRNUMBER = {1545629},
       DOI = {10.1007/BF01180430},
       URL = {https://doi.org/10.1007/BF01180430},
}

@book {MR346969,
    AUTHOR = {Kagan, A. M. and Linnik, Yu.\ V. and Rao, C. Radhakrishna},
     TITLE = {Characterization problems in mathematical statistics},
    SERIES = {Wiley Series in Probability and Mathematical Statistics},
      NOTE = {Translated from the Russian by B. Ramachandran},
 PUBLISHER = {John Wiley \& Sons, New York-London-Sydney},
      YEAR = {1973},
     PAGES = {xii+499},
   MRCLASS = {62E10},
  MRNUMBER = {346969},
}

@article {MR773445,
    AUTHOR = {Klebanov, L. B. and Maniya, G. M. and Melamed, I. A.},
     TITLE = {A problem of {V}. {M}. {Z}olotarev and analogues of infinitely
              divisible and stable distributions in a scheme for summation
              of a random number of random variables},
   JOURNAL = {Teor. Veroyatnost. i Primenen.},
  FJOURNAL = {Akademiya Nauk SSSR. Teoriya Veroyatnoste\u i\ i ee
              Primeneniya},
    VOLUME = {29},
      YEAR = {1984},
    NUMBER = {4},
     PAGES = {757--760},
      ISSN = {0040-361X},
   MRCLASS = {60E07 (60G50)},
  MRNUMBER = {773445},
MRREVIEWER = {Gerd\ Christoph},
}

@book {MR189085,
    AUTHOR = {Linnik, Yu.\ V.},
     TITLE = {Decomposition of probability distributions},
 PUBLISHER = {Dover Publications, Inc., New York; Oliver and Boyd Ltd.,
              Edinburgh-London},
      YEAR = {1964},
     PAGES = {xii+242},
   MRCLASS = {60.20},
  MRNUMBER = {189085},
}

@book {MR428382,
    AUTHOR = {{Linnik, Ju.\ V. and Ostrovs'ki\u i, \u I. V.}},
     TITLE = {Decomposition of random variables and vectors},
    SERIES = {Translations of Mathematical Monographs},
    VOLUME = {Vol. 48},
      NOTE = {Translated from the Russian},
 PUBLISHER = {American Mathematical Society, Providence, RI},
      YEAR = {1977},
     PAGES = {ix+380},
   MRCLASS = {60E05 (60G50)},
  MRNUMBER = {428382},
}

@incollection {MR133854,
    AUTHOR = {Lukacs, Eugene},
     TITLE = {Recent developments in the theory of characteristic functions},
 BOOKTITLE = {Proc. 4th {B}erkeley {S}ympos. {M}ath. {S}tatist. and {P}rob.,
              {V}ol. {II}},
     PAGES = {307--335},
 PUBLISHER = {Univ. California Press, Berkeley-Los Angeles, Calif.},
      YEAR = {1960},
   MRCLASS = {60.20},
  MRNUMBER = {133854},
MRREVIEWER = {R.\ G.\ Laha},
}

@book {MR346874,
    AUTHOR = {Lukacs, Eugene},
     TITLE = {Characteristic functions},
   EDITION = {Second},
 PUBLISHER = {Hafner Publishing Co., New York},
      YEAR = {1970},
     PAGES = {x+350},
   MRCLASS = {60E05 (62EXX)},
  MRNUMBER = {346874},
}

@book {MR810001,
    AUTHOR = {Lukacs, Eugene},
     TITLE = {Developments in characteristic function theory},
 PUBLISHER = {Macmillan Co., New York},
      YEAR = {1983},
     PAGES = {viii+182},
      ISBN = {0-02-848550-5},
   MRCLASS = {60E10},
  MRNUMBER = {810001},
MRREVIEWER = {F.\ W.\ Steutel},
}

@book {MR3242261,
     TITLE = {The {S}cottish {B}ook},
    EDITOR = {Mauldin, R. Daniel},
   EDITION = {Second},
      NOTE = {Mathematics from the Scottish Caf\'e{} with selected problems
              from the new Scottish Book,
              Including selected papers presented at the Scottish Book
              Conference held at North Texas University, Denton, TX, May
              1979},
 PUBLISHER = {Birkh\"auser/Springer, Cham},
      YEAR = {2015},
     PAGES = {xvii+322},
      ISBN = {978-3-319-22896-9; 978-3-319-22897-6},
   MRCLASS = {00B25 (01A60 03E15)},
  MRNUMBER = {3242261},
MRREVIEWER = {Godofredo\ Iommi Amun\'ategui},
       DOI = {10.1007/978-3-319-22897-6},
       URL = {https://doi.org/10.1007/978-3-319-22897-6},
}

@inproceedings{raikov1937decomposition,
  title={On the decomposition of Poisson laws},
  author={Raikov, D},
  booktitle={Dokl. Akad. Nauk SSSR},
  volume={14},
  pages={9--12},
  year={1937}
}

@book {MR2978140,
    AUTHOR = {Schilling, Ren\'e{} L. and Song, Renming and Vondra{\v{c}}ek,
              Zoran},
     TITLE = {Bernstein functions},
    SERIES = {De Gruyter Studies in Mathematics},
    VOLUME = {37},
   EDITION = {Second},
      NOTE = {Theory and applications},
 PUBLISHER = {Walter de Gruyter \& Co., Berlin},
      YEAR = {2012},
     PAGES = {xiv+410},
      ISBN = {978-3-11-025229-3; 978-3-11-026933-8},
   MRCLASS = {60E07 (31C05 43A35 44A10 47A57 47D06 60E10 60Jxx)},
  MRNUMBER = {2978140},
MRREVIEWER = {David\ Applebaum},
       DOI = {10.1515/9783110269338},
       URL = {https://doi.org/10.1515/9783110269338},
}

@article {MR1503439,
    AUTHOR = {Schoenberg, I. J.},
     TITLE = {Metric spaces and completely monotone functions},
   JOURNAL = {Ann. of Math. (2)},
  FJOURNAL = {Annals of Mathematics. Second Series},
    VOLUME = {39},
      YEAR = {1938},
    NUMBER = {4},
     PAGES = {811--841},
      ISSN = {0003-486X,1939-8980},
   MRCLASS = {99-04},
  MRNUMBER = {1503439},
       DOI = {10.2307/1968466},
       URL = {https://doi.org/10.2307/1968466},
}

@article{steutel1990set,
  title={The set of geometrically infinitely divisible distributions},
  author={Steutel, FW},
  year={1990},
  journal={Memorandum COSOR},
  volume={9042},
  publisher={Technische Universiteit Eindhoven},
  pages={1-5}
}

@book {MR4837608,
    AUTHOR = {Stroock, Daniel W.},
     TITLE = {Probability theory, an analytic view},
   EDITION = {Third},
 PUBLISHER = {Cambridge University Press, Cambridge},
      YEAR = {2025},
     PAGES = {xxii+445},
      ISBN = {978-1-009-54900-4; [9781009549035]},
   MRCLASS = {60-01 (60B11 60F05 60G44 60G51 60J35)},
  MRNUMBER = {4837608},
}

@book {MR4573029,
    AUTHOR = {Stroock, Daniel W.},
     TITLE = {Gaussian measures in finite and infinite dimensions},
    SERIES = {Universitext},
 PUBLISHER = {Springer, Cham},
      YEAR = {[2023] \copyright 2023},
     PAGES = {xii+144},
      ISBN = {978-3-031-23121-6; 978-3-031-23122-3},
   MRCLASS = {60-02 (28C20 46G12 60B11 60G15)},
  MRNUMBER = {4573029},
MRREVIEWER = {Ramon\ van Handel},
       DOI = {10.1007/978-3-031-23122-3},
       URL = {https://doi.org/10.1007/978-3-031-23122-3},
}

@book {MR5923,
    AUTHOR = {Widder, David Vernon},
     TITLE = {The {L}aplace {T}ransform},
    SERIES = {Princeton Mathematical Series},
    VOLUME = {vol. 6},
 PUBLISHER = {Princeton University Press, Princeton, NJ},
      YEAR = {1941},
     PAGES = {x+406},
   MRCLASS = {42.4X},
  MRNUMBER = {5923},
MRREVIEWER = {J.\ D.\ Tamarkin},
}

\end{document}